\documentclass{amsart}

\usepackage[mathscr]{eucal}
\usepackage{amssymb}
\usepackage[usenames,dvipsnames]{color}
\usepackage{amsthm}
\usepackage{bbold}
\usepackage{enumerate}
\usepackage{array}
\usepackage{amsmath}

\usepackage[colorlinks=true,linkcolor={Brown},citecolor={Brown},urlcolor={Brown}]{hyperref}

\usepackage{cleveref}

\numberwithin{equation}{section}
\usepackage[all]{xy}
\SelectTips{cm}{}

\newdir{ >}{{}*!/-10pt/\dir{>}}

\swapnumbers 

\newtheorem{Thm}[equation]{Theorem}
\newtheorem*{Thm*}{Theorem}
\newtheorem{Lem}[equation]{Lemma}
\newtheorem{Cor}[equation]{Corollary}
\theoremstyle{remark}

\newtheorem{Exa}[equation]{Example}

\newtheorem{Conv}[equation]{Convention}

\newtheorem{Rem}[equation]{Remark}
\newtheorem{Rec}[equation]{Recollection}

\newtheorem*{Ack}{Acknowledgements}

\newcommand{\nc}{\newcommand}
\nc{\dmo}{\DeclareMathOperator}

\dmo{\coker}{coker}
\dmo{\colim}{colim}
\dmo{\cone}{cone}
\dmo{\End}{End}
\dmo{\Hom}{Hom}
\dmo{\id}{id}
\dmo{\Id}{Id}
\dmo{\Ind}{Ind}
\dmo{\Ker}{Ker}
\dmo{\opname}{op}
\dmo{\smallh}{h}
\dmo{\supp}{supp}
\dmo{\Spc}{Spc}
\dmo{\Spec}{Spec}

\nc{\bbZ}{\mathbb{Z}}
\nc{\cat}[1]{\mathscr{#1}}
\nc{\EndHere}{\bibliographystyle{alpha}\bibliography{TG-articles}\end{document}}
\nc{\Endcat}[1]{\End_{\cat #1}}
\nc{\Homcat}[1]{\Hom_{\cat #1}}
\nc{\ie}{{\sl i.e.}\ }
\nc{\inv}{^{-1}}
\nc{\isoto}{\overset{\sim}{\,\to\,}}
\nc{\isotoo}{\overset{\sim}{\,\too\,}}
\nc{\op}{^{\opname}}
\nc{\xto}[1]{\xrightarrow{#1}}
\nc{\oto}[1]{\overset{#1}\to}
\nc{\otoo}[1]{\overset{#1}{\,\too\,}}
\nc{\Paul}[1]{{\color{Orange}#1}}
\nc{\James}[1]{{\color{ForestGreen}#1}}
\nc{\footpaul}[1]{\Paul{\rm[\footnote{\,\Paul{#1}}]}}
\nc{\footjames}[1]{\James{\rm[\footnote{\,\James{#1}}]}}
\nc{\SET}[2]{\big\{\,#1\,\big|\,#2\,\big\}}
\nc{\too}{\mathop{\longrightarrow}\limits}
\nc{\unit}{\mathbb{1}}

\dmo{\Ab}{Ab}
\dmo{\Add}{Add}
\dmo{\coev}{coev}
\dmo{\coind}{coind}
\dmo{\Comod}{CoMod}
\dmo{\Der}{D}
\dmo{\grmod}{grmod}
\dmo{\GrMod}{GrMod}
\dmo{\Ho}{Ho}
\dmo{\incl}{incl}
\dmo{\im}{im}
\dmo{\Loc}{Loc}
\dmo{\modname}{mod}%
\dmo{\Mod}{Mod}
\dmo{\pr}{pr}
\dmo{\resname}{res}
\dmo{\rmH}{H}
\dmo{\rmL}{L}
\dmo{\Res}{Res}
\dmo{\smallperf}{perf}
\dmo{\stab}{stab}
\dmo{\SH}{SH}
\dmo{\Stab}{Stab}
\dmo{\hname}{h}

\nc{\ababs}{{\sl ab absurdo}}
\nc{\adj}{\dashv}
\nc{\adjto}{\rightleftarrows}
\nc{\aka}{{a.\,k.\,a.}\ }
\nc{\ala}{{\sl \`a la}\ }
\nc{\apriori}{{\sl a priori}}
\nc{\bbC}{\mathbb{C}}
\nc{\bbP}{\mathbb{P}}
\nc{\cf}{{\sl cf.}\ }
\nc{\CComod}{\,\text{-}\Comod}%
\nc{\Dperf}{\Der^{\smallperf}}
\nc{\eg}{{\sl e.g.}}
\nc{\gm}{\mathfrak{m}}
\nc{\gp}{\mathfrak{p}}
\nc{\hook}{\hookrightarrow}
\nc{\ideal}[1]{\langle #1\rangle}
\nc{\ihom}{{\mathsf{hom}}} 
\nc{\ihomcat}[1]{\ihom_{\cat #1}}
\nc{\into}{\mathop{\rightarrowtail}}
\nc{\ind}{\textrm{ind}}
\nc{\loccit}{{\sl loc.\ cit.}}
\nc{\mmod}{\text{-\!}\modname}%
\nc{\ggrmod}{\text{-\!}\grmod}%
\nc{\modd}{\modname\text{\!-}}%
\nc{\Modd}{\Mod\text{\!-}}
\nc{\MMod}{\text{-\!}\Mod}
\nc{\onto}{\mathop{\twoheadrightarrow}}
\nc{\sminus}{\!\smallsetminus\!}
\nc{\potimes}[1]{^{\otimes #1}}
\nc{\qquadtext}[1]{\qquad\textrm{#1}\qquad}
\nc{\quadtext}[1]{\quad\textrm{#1}\quad}
\nc{\res}{\resname}
\nc{\restr}[1]{_{|_{\scriptstyle #1}}}
\nc{\sbull}{{\scriptscriptstyle\bullet}}
\nc{\sstab}{\,\text{--}\stab}%
\nc{\SpcK}{\Spc(\cat K)}
\nc{\SpcTc}{\Spc(\cat T^c)}
\nc{\then}{$\Rightarrow$}

\nc{\bbe}{\mathbb{e}}
\nc{\bbf}{\mathbb{f}}

\nc{\ohook}[1]{\,\overset{#1}\hook\,}
\nc{\oonto}[1]{\,\overset{#1}\onto\,}

\nc{\MS}{\Modd\cat S^c}
\nc{\mS}{\modd\cat S^c}
\nc{\CS}{\cat C(\cat S^c)}
\nc{\MT}{\Modd\cat T^c}
\nc{\MF}{\Modd\cat F^c}
\nc{\mT}{\modd\cat T^c}
\nc{\ML}{\MMod\cat{L}}
\nc{\mL}{\mmod\cat{L}}
\nc{\bL}{\cat B_0(\cat L)}
\nc{\BL}{\cat B(\cat L)}
\nc{\cL}{\cat C_0(\cat L)}
\nc{\CL}{\cat C(\cat L)}
\nc{\cA}{\cat{A}}
\nc{\cB}{\cat{B}}
\nc{\cC}{\cat{C}}
\nc{\cD}{\cat{D}}
\nc{\cF}{\cat{F}}
\nc{\cJ}{\cat{J}}
\nc{\cK}{\cat{K}}
\nc{\cP}{\cat{P}}
\nc{\cS}{\cat{S}}
\nc{\cT}{\cat{T}}
\nc{\cU}{\cat{U}}
\nc{\EB}{E_{\cB}}
\nc{\EBi}{E_{\cB_i}}
\nc{\EBp}{E_{\cB'}}
\nc{\EC}{E_{\cC}}
\nc{\AB}{\bar{\cA}_{\cB}}
\nc{\ABp}{\bar{\cA}_{\cB'}}
\nc{\AC}{\bar{\cA}_{\cC}}
\nc{\hB}{\boneda_{\cB}}
\nc{\hBp}{\boneda_{\cB'}}
\nc{\hC}{\boneda_{\cC}}
\nc{\hEB}{\hat E_{\cB}}
\nc{\hEBp}{\hat E_{\cB'}}
\nc{\hEC}{\hat E_{\cC}}
\nc{\bEB}{\bar E_{\cB}}
\nc{\bEBp}{\bar E_{\cB'}}
\nc{\bEC}{\bar E_{\cC}}
\nc{\fp}{^{\textrm{fp}}}
\nc{\tensid}{$\otimes$-ideal}
\nc{\tensids}{$\otimes$-ideals}
\nc{\tensnil}{$\otimes$-nilpotent}
\nc{\yoneda}{\mathbb{h}}
\nc{\boneda}{\bar\yoneda}
\nc{\bat}[1]{\bar{\cat{#1}}}
\nc{\Loct}{\Loc^{\otimes}}
\nc{\ihombat}[1]{\ihom_{\bat #1}}
\nc{\Hombat}[1]{\Hom_{\bat #1}}
\nc{\bunit}{\bar\unit}
\nc{\hunit}{\hat\unit}
\nc{\vcorrect}[1]{{\vphantom{\vbox to #1em{}}}}
\nc{\Spch}{\Spc^{\smallh}}
\nc{\SpcH}{\Spc^{\bigH}}
\nc{\SpchT}{\Spch(\cat{T}^c)}
\nc{\SpcHT}{\SpcH(\cat{T}^c)}
\nc{\SpchS}{\Spch(\cat{S}^c)}
\nc{\SpcHS}{\SpcH(\cat{S}^c)}
\nc{\SpcT}{\Spc(\cat{T}^c)}
\nc{\SpcS}{\Spc(\cat{S}^c)}
\nc{\Supph}{\Supp^{\smallh}}
\nc{\SuppH}{\Supp^{\bigH}}
\nc{\SuppBF}{\Supp_{{\rm BF}}}
\nc{\Snil}{\Supp^{\smallnil}}
\nc{\tristars}{\begin{center}*\ *\ *\end{center}}

\begin{document}


\title[Computing homological residue fields]{Computing homological residue fields\\ in algebra and topology}
\author{Paul Balmer}
\author{James C.\ Cameron}
\date{\today}

\address{Paul Balmer, UCLA Mathematics Department, Los Angeles CA 90095-1555, USA}
\email{balmer@math.ucla.edu}
\urladdr{http://www.math.ucla.edu/$\sim$balmer}

\address{James C.\ Cameron, UCLA Mathematics Department, Los Angeles CA 90095-1555, USA}
\email{jcameron@math.ucla.edu}
\urladdr{http://www.math.ucla.edu/$\sim$jcameron}

\begin{abstract}
We determine the homological residue fields, in the sense of tensor-triangular geometry, in a series of concrete examples ranging from topological stable homotopy theory to modular representation theory of finite groups.
\end{abstract}

\subjclass[2010]{18D99; 20J05, 55U35}
\keywords{Tensor-triangular geometry, homological residue field}

\thanks{First named author supported by NSF grant~DMS-1901696.}

\maketitle


\section{Introduction}


One important question in tensor-triangular geometry is the existence of residue fields. Just as we can reduce some problems in commutative algebra to problems in linear algebra by passing to residue fields, we would like to study general tensor-triangulated categories via tensor-triangulated fields (in the sense of \Cref{Rec:BKS}). Indeed, in key examples of tensor-triangulated categories such residue fields do exist: In the stable homotopy category we have Morava $K$-theories, in the stable module category of a finite group scheme over a field we can consider $\pi$-points, and of course for the derived category of a commutative ring we have ordinary residue fields.

Even though residue fields exist in those examples, at the moment there is no tensor-triangular construction of them and it is not known if they always exist, beyond very special cases~\cite{Mathew17a}. It is not even known exactly what properties one should expect from a residue field functor~$F\colon\cT\to \cF$ from our tensor-triangulated category of interest~$\cT$ to such a tensor-triangulated field~$\cF$. In particular, some of the examples above fail to give symmetric tensor functors or tensor functors at all.

The recent work~\cite{BalmerKrauseStevenson19,Balmer20a,Balmer20b} introduced and explored \emph{homological residue fields} as an alternative that exists in broad generality and is always tensor-friendly. They consist of symmetric monoidal homological functors
\begin{equation}
\label{eq:hB}%
\hB\colon\cT\to \AB
\end{equation}
from~$\cT$ to `simple' abelian categories~$\AB$. One such functor exists for each so-called \emph{homological prime}~$\cB$, as recalled in \Cref{sec:background}. These homological residue fields collectively detect the nilpotence of maps~\cite{Balmer20a} and they give rise to a support theory for not necessarily compact objects~\cite{Balmer20b}.

Homological residue fields are undeniably useful but they are defined in a rather abstract manner and it is not clear how they relate to the tensor-triangulated residue fields $F\colon\cT\to\cF$ that partially exist in examples. For each homological prime~$\cB$, there is a canonical pure-injective object $\EB$ of $\cT$ that completely determines the homological residue field~\eqref{eq:hB}. Given a functor $F\colon \cT \to \cF$ to a tensor-triangulated field satisfying mild assumptions, there is a corresponding homological prime~$\cB$. (See \Cref{lem:main-diagram}.) In the standard examples of such~$F\colon\cT\to \cF$, it is desirable to determine the corresponding pure-injective~$\EB$ and residue category~$\AB$. Our first contribution is to show that $\EB$ is a direct summand of~$U(\unit)$, where $U$ is right adjoint to~$F$. (The existence of~$U$ is a mild assumption, by Brown Representability.)

\begin{Thm*}[\Cref{thm:EB<U(1)}]
Let $F\colon \cT \to \cF$ be a monoidal exact functor to a tensor-triangulated field and suppose that~$F$ admits a right adjoint~$U$. Then the pure-injective object $\EB \in \cT$ is a direct summand of~$U(\unit)$.
\end{Thm*}

In the examples of the topological stable homotopy category and of the derived category of a commutative ring, this $U(\unit)$ is respectively a Morava $K$-theory spectrum and a residue field. These are indecomposable objects. So we conclude that for these examples the potentially mysterious $\EB$ coincides with the better understood ring object~$U(\unit)$. See \Cref{cor:EB-Morava} and~\Cref{cor:EB-ring}. The latter generalizes the case of a discrete valuation ring obtained in~\cite[Example~6.23]{BalmerKrauseStevenson19}.

In the case of the stable module category of a finite group scheme, there is a complication because the best candidates for tensor-triangulated fields -- associated to $\pi$-points -- are almost never monoidal functors for the tensor product coming from the group Hopf algebra structure. However, in the case of the stable module category of an elementary abelian $p$-group with the $p$-restricted Lie algebra tensor product structure, $\pi$-points do give tensor functors and we show in \Cref{cor:EB-pipoints} that in this case the pure-injective $\EB$ is again determined by the~$\pi$-point.

\smallbreak

In addition to understanding the pure-injective object $\EB$ we would like to describe the abelian category~$\AB$ in which a homological residue field takes its values. When $\cB$ arises from a tensor-triangulated field $F\colon\cT\to\cF$ we would like to relate $\AB$ and~$\cF$. Our second contribution is to do just this. The adjunction $F \dashv U$ gives a comonad $FU$ on~$\cF$. This comonad $FU$ then gives a comonad $\widehat{FU}$ on the abelian category of (functor) $\cF^c$-modules, where $\cF^c$ denotes the compact objects of~$\cF$; see details in \Cref{sec:background}. The Eilenberg-Moore category of comodules (\aka coalgebras) for this comonad~$\widehat{FU}$ is precisely the homological residue field~$\AB$:

\begin{Thm*}[\Cref{thm:AB-as-comodules}]
Let $F\colon \cT \to \cF$ be a monoidal functor with right adjoint~$U$, where $\cF$ is a tensor-triangulated field. The category of comodules for the comonad $\widehat{FU}$ on the functor category of $\cF^c$-modules is equivalent to the homological residue field $\AB$ corresponding to~$F$.
\end{Thm*}

In cases where the tensor-triangulated field $\cF$ in question is semisimple, such as the examples in topology and commutative algebra, this abelian category of~$\cF^c$-modules is just $\cF$ itself. So the homological residue field $\AB$ is the category of comodules for an explicit comonad on~$\cF$ itself.
Let us rephrase this result in heuristic terms. When the residue field functor $F\colon\cT\to \cF$ takes values in a semisimple triangulated category~$\cF$ (\ie one that is also abelian), then the abstractly constructed~$\AB$ contains more information than~$\cF$. There is a faithful exact functor $\AB\to \cF$ but objects of~$\AB$ `remember' more information than just being an object of~$\cF$, they remember that they come from~$\cT$ via~$F$. This additional information is encoded in the comodule structure with respect to~$FU\colon\cF\to\cF$.

\smallbreak
\Cref{sec:background} sets up the stage. We prove our two theorems in Sections~\ref{sec:EB-examples} and~\ref{sec:comodules}.

\begin{Ack}
We thank Jacob Lurie for an interesting discussion, that made us realize that the monadic adjunction of~\cite[\S\,6]{BalmerKrauseStevenson19} also satisfied \emph{co}monadicity. We also thank Tobias Barthel, Ivo Dell'Ambrogio and Greg Stevenson for useful comments.
\end{Ack}


\section{Background on homological residue fields}
\label{sec:background}%


In this section we recall some definitions and properties of the category of modules of a tensor-triangulated category and of homological residue fields. By a \emph{big tt-category}~$\cT$, we mean a compact-rigidly generated tensor-triangulated category, as in~\cite{BalmerKrauseStevenson19}. This means that $\cT$ admits all small coproducts, that its compact objects (those $x$ such that $\Homcat{T}(x,-)$ commutes with coproducts) and rigid objects (those~$x$ that admit a dual~$x^\vee$ for~$\otimes$) coincide and form an essentially small subcategory denoted~$\cT^c$ that moreover generates~$\cT$ as a localizing category.

\begin{Rec}
\label{Rec:BKS}%
For $\cT$ a big tt-category, the (functor) category of $\cT^c$-modules $\cA=\MT$ is the category of additive functors from $(\cT^c)\op$ to abelian groups. This is a Grothendieck category, hence admits enough injectives. The subcategory of finitely presented objects of $\cA$ is denoted $\cA\fp=\mT$. We have a restricted Yoneda functor $\yoneda\colon \cT \to \cA$ defined by $\yoneda(X)(-)=\hat X(-):=\Homcat{T}(-,X)\restr{\cT^c}$ for every~$X\in\cT$. This gives a commutative diagram
\[
\xymatrix@H=1em@R=1.2em{
\cT^c {}_{\vphantom{j}} \ \ar@{^{(}->}[d] \ar@{^{(}->}[r]^-{\yoneda}
& \cA\fp \ar@{^{(}->}[d] {}_{\vphantom{j}}
\\
\cT \ar[r]^-{\yoneda} & \cA
}
\]
whose first row is the usual Yoneda embedding for~$\cT^c$.

For $\cB$ a \emph{homological prime}, \ie a maximal Serre tensor-ideal of~$\cA\fp$, we obtain the homological residue field $\AB$ as the Gabriel quotient $Q\colon \cA\onto\AB:=\cA / \Loc(\cB)$. In that quotient~$\AB$ we write $\bar X$ instead of~$Q(\hat X)$, for~$X\in\cT$.
In $\AB$ we can take the injective hull of the unit and by~\cite[\S\,3]{BalmerKrauseStevenson19} this is of the form $\overline{\EB}$ for a canonical pure-injective $E_{\cB}$ of~$\cT$. Furthermore we have $\Loc(\cB)=\Ker(\hEB\otimes-)$ in~$\cA$.

A \emph{tt-field} is a big tt-category $\cF$ such that every object of $\cF$ is a coproduct of compact objects and such that tensoring with any object is faithful. Equivalently by~\cite[Theorem 5.21]{BalmerKrauseStevenson19}, a big tt-category $\cF$ is a tt-field if and only if for every non-zero $X\in \cF$ the internal hom functor $\hom_{\cF}(-,X)\colon \cF\op\to\cF$ is faithful.
\end{Rec}

We now summarize and mildly generalize some results from~\cite{BalmerKrauseStevenson19,Balmer20b}.
\begin{Lem}\label{lem:main-diagram}
Given a big tensor-triangulated category $\cT$, a tensor-triangulated field~$\cF$, and a monoidal exact functor $F\colon \cT \to \cF$ with right adjoint~$U$, we have the following diagram:
\begin{equation}\label{eq:main}
\vcenter{\xymatrix@H=1.2em@R=1.2em{
 \cat T \ar[r]^-{\yoneda}  \ar@<-.2em>[dd]_-{F}  & \MT=\cA \ar@<-.2em>[dd]_-{\hat{F}} \ar@<-.2em>[dr]_-{Q}  & \\
 & &  \MT/ \Ker(\hat{F})=\AB \ar@<-.2em>[ul]_-{R}  \ar@<-.2em>[dl]_-{\bar{F}} \\
 \cat F  \ar[r]^-{\yoneda} \ar@<-.2em>[uu]_-{U} & \MF \ar@<-.2em>[uu]_-{\hat{U}} \ar@<-.2em>[ur]_-{\bar{U}} &
}}
\end{equation}
in which $\hat{F}$ is the exact cocontinuous functor induced by $F$, the functor $Q$ is the Gabriel quotient with respect to~$\Ker(\hat F)$ and the functor $\bar{F}$ is induced by the universal property, hence $\hat F=\bar F\,Q$ and $\bar{F}$ is exact and faithful. We have adjunctions $F \dashv U$, $\hat{F} \dashv \hat{U}$, $\bar{F} \dashv \bar{U}$, and $Q \dashv R$, as depicted. Also, $\hat F\,\yoneda=\yoneda\,F$ and $\hat U\,\yoneda=\yoneda\,U$.

We have that $\cB:=\Ker(\hat{F})\cap\cA\fp$ is a homological prime and $\Ker(\hat F)=\Loc(\cB)$. Therefore $\hB=Q\circ\yoneda\colon \cT\to\AB$ is a homological residue field of~$\cT$. \end{Lem}

\begin{proof}
This is  a more general version of~\cite[Diagram~6.19]{BalmerKrauseStevenson19}, whose construction works without assuming~$F$ symmetric monoidal, just monoidal.
We do not need Condition~(3) in~\cite[Hypothesis~6.1]{BalmerKrauseStevenson19} either, as this was later used only to guarantee faithfulness of~$\bar{U}$, which we do not use here.
Specifically, see~\cite[Construction 6.10]{BalmerKrauseStevenson19}, for the relations $\hat F\,\yoneda=\yoneda\,F$ and $\hat U\,\yoneda=\yoneda\,U$. The remaining relations, involving~$\cB$ and $\AB$, can be found in~\cite[Proposition~6.17 and Construction 6.18]{BalmerKrauseStevenson19}.
The fact that $\cB$ is a homological prime can be found in~\cite[Appendix~A]{Balmer20b}.
Note that $U$ is cocontinuous because $F$ preserves compact (\ie rigid) objects. This gives~$\hat U$ and $\widehat{FU}\cong\hat F\,\hat U$.
\end{proof}

\begin{Rem}
It is important that $F$ is only assumed to be monoidal but not symmetric monoidal. In other words $F(X\otimes Y)$ is naturally isomorphic to~$F(X) \otimes F(Y)$, but this isomorphism is not necessarily compatible with the braiding. In some of our examples the functor $F$ cannot be made symmetric monoidal but it is monoidal. The reason we ask that $F$ is monoidal is so that the kernel of $\hat{F}$ is a tensor-ideal. The monoidality of $F$ is a reasonable condition that guarantees this but because $F$ is \emph{not} monoidal in other examples (\eg\ in modular representation theory) there may be better hypotheses to be discovered.
\end{Rem}


\section{The pure-injective $\EB$ in examples}
\label{sec:EB-examples}%


We now would like to study the pure-injective object $\EB$ that determines the homological residue field corresponding to~$F\colon \cT\to \cF$ as in \Cref{sec:background}.

\begin{Thm}
\label{thm:EB<U(1)}%
With the hypotheses as in \Cref{lem:main-diagram}, the pure-injective object $U(\unit)$ in~$\cT$ admits~$\EB$ as a direct summand.
\end{Thm}

\begin{proof}
That $U(\unit)$ is pure-injective is~\cite[Lemma 6.12]{BalmerKrauseStevenson19}. It corresponds to the injective $\bar{U}(\unit)=\overline{U(\unit)}$ in~$\AB$.

By the $\bar{F} \dashv \bar{U}$ adjunction we have that $\bar{F}(\unit) \to \bar{F}\bar{U}\bar{F}(\unit)$ is a monomorphism. As $\bar{F}$ is faithful, we deduce that the unit $\unit \to \bar{U} \bar{F}(\unit)=\bar{U}(\unit)$ is a monomorphism in~$\AB$. But $\unit\to\bEB$ is the injective hull in~$\AB$. Hence $\bar{U}(\unit)=\bEB \oplus \bar{X}$ for some pure-injective~$X\in\cT$.

Recall by \cite[Corollary~2.18\,(c)]{BalmerKrauseStevenson19} that for any injective $\bar I$ in $\AB$ and any $X\in \cT$ we have that $\Hom_{\cT}(X,I)\cong\Hom_{\AB} (\bar{X},\bar{I})$. So, since $\bar{U}(\unit)$ has $\bEB$ as a summand in $\AB$, it follows that $U(\unit)$ also has $\EB$ as a summand in~$\cT$.
\end{proof}

Now, in several examples $U(\unit)$ is a familiar object and it does not have any non-trivial summands.

\begin{Lem}\label{lem:indecomposability}
Suppose that $\cT$ is a tensor-triangulated category that is generated by its~$\otimes$-unit~$\unit$. For $E \in \cT$ if $\pi_*E:=\Hom(\Sigma^*\unit,E)$ is indecomposable as a $\pi_*\unit$-module then $E$ is indecomposable in~$\cT$.
\end{Lem}
\begin{proof}
We have $\pi_*(E'\oplus E'')\cong \pi_*E' \oplus \pi_*E''$ as $\pi_*\unit$-modules. So the result follows immediately from $\pi_*$ being a conservative functor when $\cT$ is generated by~$\unit$.
\end{proof}

\begin{Cor}
\label{cor:EB-ring}%
Let $R$ be a ring with derived category $\cT=\Der(R)$ and let $\cB$ be a homological prime corresponding to the prime $\gp$ of $\Spec(R)\cong \Spc(\Dperf(R))$. Then the pure-injective object $\EB$ is isomorphic to the residue field~$\kappa(\gp)[0]$.
\end{Cor}

\begin{proof}
The homological prime $\cB$ is obtained as in the setting of \Cref{thm:EB<U(1)} from the residue field functor $F\colon\Der(R) \to \Der(\kappa(\gp))$.  By the above discussion, we just need to show that $U(\unit)=\kappa(\gp)[0]$ does not split in $\Der(R)$. By \Cref{lem:indecomposability} this follows from the fact that $\kappa(\gp)$ does not split as an $R$-module.
\end{proof}

Morava $K$-theories furnish all of the tt-primes in the stable homotopy category $\SH$ by~\cite{HopkinsSmith98}. Indeed, each tt-prime of~$\SH^c$ is the kernel $\Ker(\SH^c\to K(p,n)_*\text{-}\GrMod)$ of a functor given by $X \mapsto K(p,n)_*X$ where $p$ is a prime and $n\ge 0$ a natural number or~$\infty$. We write $K(p,\infty)$ for~$\rmH\bbZ/p$ and~$K(p,0)$ for~$\rmH \mathbb{Q}$. When $0<n<\infty$, we have $K(p,n)_*=\mathbb{F}_{\!p}[v_n^{\pm1}]$ with $v_n$ in degree~$2(p^n-1)$. We denote by $K(p,n)_*\text{-}\GrMod$ the category of graded $K(p,n)_*$-modules; this is a semisimple triangulated category.

\begin{Rem}
Each spectrum $K(p,n)$ is a  ring spectrum but they are not $E_{2}$ except for $n=0$ and $n=\infty$; see \cite[Corollary 5.4]{AntolinBarthel19} for a proof of this folklore result. The $K(2,n)$ are not even commutative rings in the homotopy category except for $n=0$ and $n=\infty$ \cite{Wurgler86}. Nevertheless the homotopy category of $K(p,n)$-modules is equivalent via taking homotopy groups to the symmetric monoidal category $K(p,n)_*\text{-}\GrMod$.
By the K{\"u}nneth isomorphism the functor $\SH \to K(p,n)_*\text{-}\GrMod$ is monoidal, although not necessarily \emph{symmetric} monoidal for $p=2$ and~$0<n<\infty$.
\end{Rem}

Because $K(p,n)_*$ is a graded field, the category $K(p,n)_*\text{-}\GrMod$ is a tensor-triangulated field. So for each Morava $K$-theory we have a homological residue field $\AB$ and a pure-injective $\EB$ that we want to identify.

\begin{Lem}\label{lem:Knindecomposable}
Each Morava $K$-theory spectrum $K(p,n)$ is indecomposable in the stable homotopy category.
\end{Lem}

\begin{proof}
This follows from~\cite[Proposition 1.10]{HopkinsSmith98}, which states that any retract in $\SH$ of a $K(p,n)$-module is a wedge of shifts of~$K(p,n)$. Because $K(p,n)_*$ is either $\mathbb{F}_p$ or $0$ in each dimension, it follows that $K(p,n)$ itself is indecomposable in~$\SH$.
\end{proof}

\begin{Cor}
\label{cor:EB-Morava}%
For $\cB$ the homological prime of~$\Spch(\SH^c)$ corresponding to~$K(p,n)$, we have an isomorphism~$\EB \simeq K(p,n)$ in~$\SH$.
\end{Cor}

\begin{proof}
Because $\EB$ is constructed as in \Cref{lem:main-diagram} from the monoidal functor $F\colon \SH \to \Ho(K(p,n)\MMod)$,  by \Cref{thm:EB<U(1)} we have that $\EB$ is a summand of~$K(p,n)$. Therefore by \Cref{lem:Knindecomposable} we have that $\EB$ is isomorphic to~$K(p,n)$.
\end{proof}

\begin{Rem}
The tt-primes of the $G$-equivariant stable homotopy category~$\SH(G)$ have been determined for $G$ finite in~\cite{BalmerSanders17} and for general compact Lie groups~$G$ in~\cite{BarthelGreenleesHausmann20}. They are all pulled-back from the chromatic ones in~$\SH$ via geometric fixed-point functors $\Phi^H\colon\SH(G)\to \SH$ for (closed) subgroups~$H\le G$. The functor~$\Phi^H$ admits a right adjoint that we denote $U^H\colon \SH\to \SH(G)$. It follows that the pure-injective $\EB$ corresponding to~$H$ and $K(p,n)$ is a summand of~$U^H(K(p,n))$. When $H=G$ it is easy to show that these objects remain indecomposable in~$\SH(G)$ but we do not know whether this holds in general.
\end{Rem}

\begin{Rem}
We have shown that in two of the main examples, the derived category of a ring and the stable homotopy category, the $\EB$ are ring objects. It is natural to wonder if this is always the case. By~\cite{BalmerKrauseStevenson19}, there is always a map $\EB\otimes \EB\to \EB$ retracting $\EB\otimes\eta$ where $\eta\colon \unit\to\EB$ comes from the definition of~$\bEB$ as the injective hull of~$\bunit$. Those `weak rings' $\EB$ would be actual rings if we always had $\EB=U(\unit)$ but the latter is not necessarily true. Indeed, one can `overshoot' the mark as explained in the following example.
\end{Rem}

\begin{Exa}
\label{Exa:over-shoot}%
For a field extension $K \to L$, we have a tt-functor on derived categories $\cT:=\Der(K) \to \Der(L)=:\cF$. Note that $\cT$ is already a tt-field itself. The pure-injective $\EB$ associated to the only homological prime $\cB=0$ of~$\cT$ is the $\otimes$-unit $\unit_{\cT}=K[0]$ itself, whereas $U(\unit_{\cF})$ is~$L[0]$ and of course $L\simeq K^{\dim_K(L)}$.
\end{Exa}

Our final example is that of $\pi$-points in modular representation theory. Let $k$ be a field of characteristic $p$ dividing the order of a finite group~$G$. Recall~\cite{FriedlanderPevtsova07} that a $\pi$-point of the group ring $kG$ is a flat algebra homomorphism $\pi\colon KC_p \to KG$ which factors as $KC_p \to KA \to KG$ for $A$ an elementary abelian $p$-subgroup of $G$ and $KA \to KG$ the map induced by the inclusion~$A\le G$, and where $K$ is a field extension of $k$. There is a corresponding functor $\pi^*\colon\Stab{kG} \to \Stab{KC_p}$ composed of extension-of-scalars to~$K$ followed by restriction along$~\pi$. This functor acts like a residue field functor, but frustratingly this functor is not monoidal, and so \Cref{lem:main-diagram} cannot be used to construct a homological residue field.

Two $\pi$-points $\gamma: KC_p \to KG$, $\zeta:LC_p \to LG$ for $kG$ are equivalent if for every finite dimensional $kG$-module $M$, $\gamma^*( M)$ is trivial in $\Stab{KC_p}$ if and only if $\zeta^*(M)$ is trivial in $\Stab{L C_p}$.  By Lemma 2.2 of \cite{FriedlanderPevtsova05}, see also example 1.6 of \cite{BensonIyengarKrausePevtsova18}, every $\pi$-point $KC_p\cong K[t]/t^p \to K[x_1,\dots,x_n]/(x_i^p)\cong KA$ is equivalent to one of the form $t \mapsto \sum_{i=1}^n \alpha_i x_i$ for ~$\alpha=(\alpha_1,\ldots,\alpha_n)\in K^n\sminus\{0\}$, and $\alpha, \beta \in K^n\sminus \{0\}$  give equivalent $\pi$-points if and only if they give the same point of $\bbP^{n-1}_K$.

We will denote such a map by $\pi_{\alpha}$. Up to equivalence, one indexes the $\pi$-points $\pi_{\alpha}$ by the closed points of $\bbP^{n-1}_{K}\simeq\Spc (\Stab(KA)^c)$. Allowing all field extensions $K/k$, the $\pi_{\alpha}$ parameterize all points of $\Spc(\Stab(kA)^c)$.

The map $\pi_{\alpha}$ is a Hopf algebra map when the source and target have the Hopf algebra structure coming from thinking of $KC_p$ and $KA$ as the restricted  enveloping algebra of a $p$-restricted Lie algebra. Consequently the functor $\pi_{\alpha}^*\colon \Stab(KA) \to \Stab(KC_p)$ is a monoidal functor with respect to this different tensor.
In formulas, this corresponds to the comultiplication $\Delta$ on $KA=K[x_i,\dots,x_n]/(x_i^p)$ given by $\Delta(x_i)=x_i \otimes 1+ 1 \otimes x_i$. See \cite{CarlsonIyengar17} for a discussion of these structures.

\begin{Conv}
\label{conv:Lie}%
From now on we use the  Hopf algebra structure on $KA$ coming from the $p$-restricted Lie algebra structure to put a tensor structure on~$\Stab(KA)$.
\end{Conv}

The $\pi_{\alpha}^*\colon\Stab(kA)\to\Stab(KC_p)$ are now tensor-triangular residue field functors. So, there are corresponding homological residue fields $\Stab(kA) \to \AB$ and corresponding pure-injectives $\EB$ in~$\Stab(kA)$. We would like to relate the $\pi$-point to the pure-injective object. First, we note what $\pi$-points do on cohomology.

\begin{Lem}\label{lem:pipoints-cohomology}
Let $\pi_{\alpha}\colon kC_p\to kA$ be a $k$-rational $\pi$-point. For $p$ odd, the induced map on cohomology $\pi_{\alpha}^*\colon \rmH^*(A,k)\cong k[\eta_1,\dots,\eta_n] \otimes \Lambda(\xi_1,\dots,\xi_n) \to k[\zeta] \otimes \Lambda(\omega)\cong \rmH^*(C_p,k)$ is given by $\pi_{\alpha}^*(\eta_i)=\alpha_i^p \zeta$ and $\pi_{\alpha}^*(\xi_i)=\alpha_i\omega$ for all~$i=1,\ldots,n$. For $p=2$, there is no exterior power and $\pi_{\alpha}^*(\eta_i)=\alpha_i \zeta$ for all~$i=1,\ldots,n$.
\end{Lem}

\begin{proof}
For the polynomial classes this is~\cite[Propositions~2.20 and~2.22]{Carlson83} and the result for the exterior classes follows from the proof of the first proposition.
\end{proof}

\begin{Cor}\label{cor:EB-pipoints}
Let $A$ be an elementary abelian $p$-group. See \Cref{conv:Lie}.
\begin{enumerate}[\rm(a)]
\item
\label{it:k}%
The pure-injective $\EB$ of $\Stab(kA)$ associated to a (closed) $k$-rational $\pi$-point $\pi_{\alpha}\colon kC_p \to kA$ is given by $\coind_{kC_p}^{kA}k=\Hom_{kC_p}(kA,k)$.
\item
\label{it:K}%
Let $K/k$ be an extension. The pure-injective $\EB$ of $\Stab(kA)$ associated to a $\pi$-point $\pi_{\alpha}\colon KC_p \to KA$ is a summand of $\res^{KA}_{kA}\coind_{KC_p}^{KA}K=\Hom_{KC_p}(KA,K)$.
\end{enumerate}
\end{Cor}
\begin{proof}
The second claim follows immediately from~\Cref{thm:EB<U(1)}. The first claim will follow from the second (with $K=k$) if we can show that $\coind_{kC_p}^{kA} k$ is indecomposable in~$\Stab(kA)$. From \Cref{lem:indecomposability} it is enough to show that $\pi_*( \coind_{kC_p}^{kA}k)$ is indecomposable as a $\pi_*k=\hat\rmH{}^{-*}(A,k)$ module. But $\pi_*( \coind_{kC_p}^{kA} k)=\hat\rmH{}^{-*}(C_p,k)$ with the module structure induced by the restriction map $\pi^*_{\alpha}\colon \hat\rmH{}^*(A,k) \to \hat\rmH{}^*(C_p,k)$ and this module is indecomposable by~\Cref{lem:pipoints-cohomology}.
\end{proof}

\begin{Rem}
In~\eqref{it:K}, there is no reason for the object~$\res^{KA}_{kA}\coind_{KC_p}^{KA}K$ to be indecomposable, for one can `overshoot' the right residue field as in \Cref{Exa:over-shoot}.
\end{Rem}

\begin{Rem}
Because $\cK=\Stab(kA)^c$ is generated by the unit, all thick subcategories are tensor-ideals, and hence the two different tensor structures on $\cK$ give the same tensor-triangular spectrum. However, this does not imply that the homological spectra are in canonical bijection. In other words it is not necessarily the case that if $I$ is a Serre tensor-ideal of $\modd\cK$ in the restricted Lie algebra tensor structure then it is also a Serre tensor-ideal in the group tensor algebra structure.
Therefore these results do not give classifications of the objects $\EB$ of $\Stab(kA)$ with the group tensor structure, \apriori. Further research might tell.
\end{Rem}


\section{Homological residue fields as comodules}
\label{sec:comodules}%


Recall Diagram~\eqref{eq:main}. It is proved in~\cite{BalmerKrauseStevenson19} that $\bar{U}$ is monadic under the assumption that $U$ is faithful, \ie that $F$ is surjective-up-to-summands (an assumption we do not make here). In that case, one can describe $\Modd \cat F^c$ as the Eilenberg-Moore category of $\bar U(\unit)$-modules in~$\AB$. This was the logic of~\cite{BalmerKrauseStevenson19}: How to recover~$\cat{F}$ from the abstract~$\AB$?

However, in many of the examples, the category~$\cF$ is more familiar than $\AB$ and we would rather like a description of $\AB$ in terms of~$\cF$. Towards this end, we will show that $\bar{F}$ is \emph{co}monadic, so $\AB$ is the Eilenberg-Moore category of comodules for the comonad $\bar{F}\bar{U}$ over $\Modd \cat F^c$.

For this, we use the dual version of the Beck monadicity theorem. Denote the Eilenberg-Moore category of comodules for a comonad $H$ on $\cat{D}$ by $H\CComod_{\cat{D}}$.

\begin{Lem}
\label{Lem:Beck}%
Suppose we have an adjunction $F\colon \cat{C} \adjto \cat{D}:U$ and that $\cat{C}$ and $\cat{D}$ have equalizers and $F$ is conservative and preserves equalizers. Then $F$ is comonadic, \ie the comparison functor $f\colon \cat{C} \to FU\CComod_\cat{D}$ is an equivalence of categories.
\end{Lem}
\begin{proof}
This follows from the usual Beck Monadicity theorem~\cite[\S\,VI.7]{MacLane98} applied to the opposite category.
\end{proof}

\begin{Thm}\label{thm:AB-as-comodules}
In~\eqref{eq:main} the functor $\bar{F}$ is comonadic, so $\AB$ is equivalent to the Eilenberg-Moore category of comodules for the comonad~$\bar{F}\bar{U}$.
\end{Thm}
\begin{proof}
The conditions of \Cref{Lem:Beck} hold since $\bar{F}$ is conservative and exact and the categories involved are abelian and hence have all equalizers.
\end{proof}

\begin{Rem}
We have $\widehat{FU}\cong\hat{F}\hat{U}\cong\bar{F}\,Q\,R\,\bar{U}\cong\bar{F}\bar{U}$. The functor $\hat{F}$ is not conservative in general, so $\hat{F}$ itself is not comonadic.
\end{Rem}

In two of the examples at hand, the stable homotopy category and the derived category of a ring, the tensor-triangulated fields~$\cF$ in the picture are semisimple triangulated categories and hence are already abelian. So, the abelian category of (functor) modules on these tensor-triangulated fields is just the tensor-triangulated field~$\cF$ itself, \ie restricted Yoneda $\yoneda\colon \cF\to \cF^c\MMod$ is an equivalence.
\begin{Cor}
\label{cor:SH}%
In the case of $\SH$, the homological residue category~$\AB$ for the homological prime~$\cB$ corresponding to a Morava $K$-theory spectrum $K(p,n)$ is equivalent to the Eilenberg-Moore category of comodules for the comonad $FU$ on the category of graded $K(p,n)_*$-modules, where $FU$ is associated to the free/forgetful adjunction $F\colon\SH\adjto \Ho(K(p,n)\MMod)\cong K(p,n)_*\text{-}\GrMod :U$.
\end{Cor}

\begin{Cor}
\label{cor:D(R)}
In the case of $\Der(R)$ for $R$ a commutative ring, the homological residue category $\AB$ for the homological prime~$\cB$ corresponding to a Zariski point $i\colon\Spec(\kappa(\gp))\to\Spec R$ is equivalent to the Eilenberg-Moore category of comodules for the comonad $\rmL i^*\circ i_*$ on graded $\kappa(\gp)$-modules, where $\rmL i^*\circ i_*$ is associated to the usual adjunction $\rmL i^*\colon\Der(R)\adjto\Der(\kappa(\gp))\cong \kappa(\gp)\text{-}\GrMod:i_*$.
\end{Cor}

\begin{proof}[Proof of Corollaries~\ref{cor:SH} and~\ref{cor:D(R)}]
By \Cref{thm:AB-as-comodules} we have that in these two cases the category~$\AB$ is equivalent to $FU$-comodules on $\cF$, where $\cF$ is respectively graded $K(p,n)_*$-modules and graded $\kappa(\gp)$-modules.

So, the claim follows from the fact that if $\cF$ is a semisimple abelian category then $\cF^c\MMod \cong \cF$ and, under this equivalence, the comonad $\widehat{FU}$ boils down to~$FU$.
\end{proof}



\end{document}